\documentclass[12pt, reqno]{amsart}
\usepackage{amsmath, amsthm, amscd, amsfonts, amssymb, graphicx, color}
\usepackage{mathrsfs}
\usepackage{newtxtext}
\usepackage[bookmarksnumbered, colorlinks, plainpages]{hyperref}
\hypersetup{colorlinks=true,linkcolor=red, anchorcolor=green, citecolor=cyan, urlcolor=red, filecolor=magenta, pdftoolbar=true}

\newtheorem{theorem}{Theorem}[section]
\newtheorem{lemma}[theorem]{Lemma}

\newtheorem{cor}[theorem]{Corollary}

\theoremstyle{remark}
\newtheorem{remark}[theorem]{\bf{Remark}}
\numberwithin{equation}{section}
\allowdisplaybreaks
\begin{document}

\title [Improved upper bounds of Berezin number \ldots]{Improved upper bounds for the Berezin numbers of operators on reproducing kernel Hilbert spaces}
\author[R. Birbonshi, S. Ghosh, F. Kittaneh S. Mahapatra, S. Ojha,  ]{ Riddhick Birbonshi, Sumon Ghosh, Fuad Kittaneh, Saikat Mahapatra, Sarita Ojha }
	\address[Birbonshi] {Department of Mathematics, Jadavpur University, Kolkata 700032, West Bengal, India}
\email{riddhick.math@gmail.com}

\address[Ghosh]{Department of Mathematics, Indian Institute of Engineering Science and Technology, Shibpur, 711103, West Bengal, India}
\email{isumonghoshmath@gmail.com}

\address[Kittaneh] {Department of Mathematics, The University of Jordan, Amman, Jordan, and Department of Mathematics, Korea University, Seoul 02841, South Korea}
\email{fkitt@ju.edu.jo}

\address[Mahapatra]{Department of Mathematics, Jadavpur University, Kolkata 700032, West Bengal, India}
\email{smpatra.lal2@gmail.com}
	
\address[Ojha]{Department of Mathematics, Indian Institute of Engineering Science and Technology, Shibpur, 711103, West Bengal, India}
\email{sarita.ojha89@gmail.com}



\begin{abstract} In this article, several upper bounds for the Berezin numbers of bounded linear operators on reproducing kernel Hilbert spaces are obtained through the use of interpolation paths of symmetric means and Orlicz functions. With suitable selections of these paths and functions, we show that
 the results presented here refine and generalize several earlier known findings. Furthermore, we derive some Berezin number inequalities for such operators using refined Young's inequalities.

\end{abstract}
\maketitle

\section{Introduction}
 Let $X$ be a non-empty set. A reproducing kernel Hilbert space $\mathcal{H}=\mathcal{H}(X)$ is a Hilbert space of complex valued functions on the set $X$, equipped with an inner product $\langle \cdot, \cdot\rangle$ and associated norm $\|\cdot\|$, with the property that, for every $x\in X$, the evaluation functional 
$\phi\rightarrow\phi(x)$ is continuous on 
$\mathcal{H}$ (see \cite{paulsen}).
  By Riesz representation theorem for each $x\in X$, there exists a unique element $k_x\in\mathcal{H}$ such that $\phi(x)=\left\langle\phi,k_x\right\rangle$ for all $\phi\in\mathcal{H}$. The collection of functions $\{k_x:x\in X\}$ is called the set of all reproducing kernels of $\mathcal{H}$ and $\{\hat{k}_x=k_x/\|k_x\|:x\in X\}$ is the set of all normalized reproducing kernels of $\mathcal{H}$. 
\par Let $A\in\mathscr{B}(\mathcal{H})$, where $\mathcal{H}$ is a reproducing kernel Hilbert space, then the Berezin transform of $A$ (see \cite{FelixA,Quantization}) is the function $\Tilde{A}$ on $X$ defined by 
 \begin{equation*}
  \Tilde{A}(x)=\big\langle A\hat{k}_x,\hat{k}_x\big\rangle~\text{for all}~x\in X.    
 \end{equation*}
The Berezin set, Berezin number and Berezin norm of $A$ (see \cite{New estimations,invertibility}) are denoted by $\textbf{Ber}(A),\textbf{ber}(A)$ and $\|A\|_\textbf{Ber}$, respectively, and defined as 
\begin{eqnarray*}
    \textbf{Ber}(A)&:=&\{\Tilde{A}(x):x\in X\},\\
    \textbf{ber}(A)&:=&\sup\{|\Tilde{A}(x)|:x\in X\}
\end{eqnarray*}
and 
\begin{equation*}
    \|A\|_\textbf{ber}:=\sup\left\{\big|\big\langle A\hat{k}_x,\hat{k}_y\big\rangle\big|:x,y\in X\right\}.
\end{equation*}
The Berezin number of an operator $A\in\mathscr{B}(\mathcal{H})$ satisfies the following properties:
\begin{enumerate}
    \item $\textbf{ber}(A) \leq \|A\|_\textbf{ber}\leq\|A\|$.
    \item $\textbf{ber}(\alpha A)=|\alpha|\textbf{ber}(A)$ for all $\alpha\in\mathbb{C}$.
    \item $\textbf{ber}(A+B)\leq\textbf{ber}(A)+\textbf{ber}(B)$ for all $A,B\in\mathscr{B}(\mathcal{H})$.
\end{enumerate}
$A\in \mathscr{B}(\mathcal{H})$ is called a positive operator if $\langle A x,x \rangle\ge0$ for all $x\in \mathcal{H}$ and it is denoted by $A\ge0$. For $A\in \mathscr{B}(\mathcal{H})$, $|A|$ stands for the positive operator $(A^*A)^{\frac{1}{2}}$.
If $A$ is positive, then $\|A\|_\textbf{ber}=\textbf{ber}(A)$ (see \cite{anirb}). A reproducing kernel Hilbert space possesses the ``Ber property'', if the Berezin transform completely characterizes the operator (see \cite{Karaev}). It is well known that the Hardy-Hilbert space is a reproducing kernel Hilbert space which has the ``Ber property'' (see \cite{Zhu}). Recall that, the Hardy-Hilbert space on the unit disc $\mathbb{D}$ is denoted by $H^2(\mathbb{D})$, and is defined as 
\begin{equation*}
 H^2(\mathbb{D})=\left\{\phi:\phi(z)=\sum\limits_{n=0}^{\infty}a_nz^n ~\text{with}~\sum\limits_{n=0}^{\infty}|a_n|^2~\text{is finite}\right\},  
\end{equation*}
(see \cite{martinez2007introduction}). Here we note that the Berezin number generally does not define a norm on $\mathscr{B}(\mathcal{H})$, unless $\mathcal{H}$ possesses the ``Ber property''. Over the years, several mathematicians have studied the Berezin number and Berezin norm inequalities of reproducing kernel Hilbert space operators (see \cite{garayev2021inequalities,hajmohamadi2020improvements,saikat,saikat2,majee2023numerical,yamanci2017numerical,yamanci2022further}).\\
A function $\sigma:[0,\infty)\times[0,\infty)\rightarrow[0,\infty)$ is said to
be mean if it fulfills the following properties (see \cite{bhatia06,Bhatia}):
\begin{enumerate}
    \item $\sigma(a,b)\geq 0$;
    \item $a\leq\sigma(a,b)\leq b$ if $a\leq b$;
    \item $\sigma(a,b)$ is monotone increasing in both $a$ and $b$;
    \item $\sigma(\alpha a,\alpha b)=\alpha\sigma(a,b)$ for all $\alpha>0$;
    \item $\sigma(a,b)$ is continuous in $a$ and $b$.
\end{enumerate}
If $\sigma(a,b)=\sigma(b,a)$, then $\sigma$ is called symmetric mean. Further, if $\sigma$ is a symmetric mean, and if for each $\lambda\in[0, 1]$, $\sigma_\lambda$ is a mean defined on $[0,\infty)\times[0,\infty)$ such that
\begin{enumerate}
    \item $\sigma_0(a,b)=a, \sigma_1(a,b)=b,\sigma_\frac{1}{2}(a,b)=\sigma(a,b)$ for all $a,b\geq0$;
    \item For all $\lambda,\mu\in[0,1],$ we have $\sigma\big(\sigma_\lambda(a,b),\sigma_\mu (a,b)\big)=\sigma_\frac{\lambda+\mu}{2}(a,b)$;
    \item For each $0\leq\lambda\leq1$, $\sigma_\lambda$ is increasing in each of its components;
\end{enumerate}
then $\sigma_\lambda$ is called an interpolational path for $\sigma$. To learn more about interpolational path, see \cite[Section 5.3]{atanu}. Here we write $a\sigma b$ instead of $\sigma(a,b)$.\\
It can be easily demonstrated that the set of all $t\in [0, 1]$ satisfying the following equation:
\begin{equation}\label{meaneq}
   (a\sigma_\alpha b)\sigma_t(a\sigma_\beta b)=a\sigma_{(1-t)\alpha+t\beta}b
\end{equation}
for every $\alpha,\beta\in [0, 1]$ and every $a, b \geq 0$, forms a convex subset of $[0, 1]$, which includes both $0$ and $1$. Consequently, the equality \eqref{meaneq} holds for all  $\alpha,\beta, t\in [0, 1]$.
If $a,b$ are two non-negative numbers, then their arithmetic mean is defined as $a\nabla b=\frac{a+b}{2}$. This mean is symmetric. An interpolation path for this mean can be written as $a\nabla_\lambda b=(1-\lambda)a+\lambda b,~0\leq\lambda\leq 1$.
For $a,b\geq 0$, the geometric mean is defined by $a\#b=\sqrt{ab}$. The corresponding interpolational path for $\#$ is $a\#_\lambda b=a^{1-\lambda}b^\lambda,~0\leq\lambda\leq 1$.
It is well known that for each $0 \le \lambda\le 1$, we have
$a\#_\lambda b\le a\nabla_\lambda b$.
 For more about the mean theory, we refer to \cite{Fujii,Furuchi}.  Recently, in \cite{mean1}, Bakherad et al. introduced several generalizations of the refined Cauchy–Schwarz inequality by employing arbitrary means. Utilizing these generalizations, several  bounds for the numerical radius and the Berezin number were subsequently derived in \cite{mean1,mean2}.
 
\par A map $f:[0,\infty)\rightarrow [0,\infty)$ which is continuous, convex, non-decreasing, $f(0)=0$ and $f(u)\rightarrow\infty$ as $u\rightarrow\infty$ is called Orlicz function (see \cite{comorlicz}). If $f(u)>0$ for all $u>0$, then the Orlicz function $f$ is called non-degenerate and if 
$f(u)=0$ for some $u>0$, then $f$ is called a degenerate Orlicz function. In this article, we consider the non-degenerate Orlicz functions. The integral representation of an Orlicz function is 
\begin{equation*}
    f(u)=\int\limits_0^up(t)dt,
\end{equation*} here $p$ is a non-decreasing function with $p(0)=0, \ p(t)>0$ for $t>0$ and $\lim\limits_{t\rightarrow\infty}p(t)=\infty$. The function $p$ is known as kernel of $f$. If $q$ be the right inverse of $p$, then $q$ is defined as $q(s)=\sup\{t:p(t)\leq s\}, ~ s\geq0$. Here, $q$ satisfies similar properties to $p$. Complementary Orlicz function $g$ to $f$ is defined by
\begin{equation*}
   g(v)=\int\limits_0^v q(s)ds.
\end{equation*}
The pair $(f,g)$ is called mutually complementary Orlicz functions. Recently, Gürdal et al. derived several Berezin number inequalities for bounded linear operators using Orlicz functions (see \cite{Gürdal}).

The purpose of this paper is to derive new upper bounds for the Berezin number of bounded linear operators on reproducing kernel Hilbert spaces. In Section \ref{sec2}, we review some well-known inequalities that are crucial in our subsequent proofs. Section \ref{sec3} is devoted to establishing several new upper bounds for the Berezin number by employing interpolation paths of symmetric means along with Orlicz functions. Through suitable choices of these paths and functions, we show that our results refine and generalize several existing results of  
\cite{pintuacta,hajmohamadi2020improvements,Taghavi}.  Lastly, Section \ref{sec4} outlines further refined Berezin number inequalities by applying improved versions of Young’s inequality, which lead to improvements of certain results previously obtained in \cite{esti,Taghavi}.

\section{\textbf{Preliminaries}}\label{sec2}

We begin this section with some auxiliary results that will be used repeatedly to obtain our main results. The first lemma is on complementary Orlicz function. 

\begin{lemma} \cite{comorlicz}\label{orlizineq} Let $f$ and $g$ be two complementary Orlicz functions. Then for any $a,b\ge 0$, the following inequality holds
\begin{eqnarray*}
   ab\le f(a)+g(b). 
\end{eqnarray*}
     
\end{lemma}
The second lemma is known as mixed Schwarz inequality.
\begin{lemma}\cite{furuta pro}\label{txy^2}
    Let $A\in \mathscr{B}(\mathcal{H})$ and let $u,v\in \mathcal{H}$. Then 
    \begin{equation*}
        |\langle Au, v\rangle|^2\leq \langle|A|^{2\alpha}u,u\rangle\langle|A^*|^{2(1-\alpha)}v,v\rangle
    \end{equation*}
    for all $0\leq \alpha \leq 1.$
\end{lemma}

The next lemma is given as:
\begin{lemma}\cite{esti}\label{estimation1}
 If $u,v,e$ are vectors in $\mathcal{H}$ and $\langle e,e \rangle=1$, then 
 \begin{equation*}
     |\langle u,e\rangle\langle e,v\rangle
     |^r\leq \frac{2\zeta+1}{2\zeta+2}\|u\|^r\|v\|^r+\frac{1}{2\zeta+2}|\langle u,v\rangle|^r,
 \end{equation*}
 for any $\zeta\geq 0$ and $r\geq 1$.
\end{lemma}

The fourth lemma is called as the classical Jensen's inequality and the spectral theorem for positive operators.
\begin{lemma}\cite{atanu}\label{vvvcc}
Let $A\in\mathscr{B}(\mathcal{H})$ be a self-adjoint operator with spectrum of $A$ is contained in $[m,M]$ for some scalars $m < M$. If $f(t)$ is a continuous and convex function on $[m,M]$, then 
\[ f\big(\langle Au,u \rangle\big)\le \langle f(A)u,u \rangle,\]
for every unit vector $u\in\mathcal{H}$.
\end{lemma}

Now, we note the following scalar inequality involving interpolation paths associated with three arbitrary symmetric means.
    If $\mu,\nu,\lambda\in[0,1]$ and $0\le a\leq b\leq c$, then
    \begin{equation}
        a\leq (a\sigma_\mu b)\rho_\lambda(a\tau_\nu c)\leq c. \label{zzxx}
    \end{equation} Using this inequality, we obtain the following lemma. 

\begin{lemma}\label{meanpower r}
     Let $u,v \in \mathcal{H}$   and $f$ be a multiplicative Orlicz function. Then  for any $\mu,\nu,\lambda\in[0,1]$ the following inequality holds
     \begin{eqnarray*}
           && f\left( |\langle u,v \rangle|^2\right)\\ &\leq&\left( f(  |\langle u,v \rangle|^2)~\sigma_\mu~ f(|\langle u,v\rangle |)f\big(\sqrt{\langle u,u\rangle\langle v,v\rangle }\big)\right)\rho_\lambda \left( f(  |\langle u,v \rangle|^2)~\tau_\nu~ f(\langle u,u\rangle\langle v,v\rangle )\right).
     \end{eqnarray*}
\end{lemma}

\begin{proof}
Let $u,v \in \mathcal{H}$. Then by the Cauchy-Schwarz inequality, we have
\begin{eqnarray*}
    |\langle u,v \rangle|^2&\le& |\langle u,v\rangle |\sqrt{\langle u,u\rangle\langle v,v\rangle }\le\langle u,u\rangle\langle v,v\rangle. \label{gg2}
\end{eqnarray*}
so,
\begin{eqnarray*}
    f\left(|\langle u,v \rangle|^2\right)&\le& f\left(|\langle u,v\rangle |\right)f\big(\sqrt{\langle u,u\rangle\langle v,v\rangle }\big)\leq f\left(\langle u,u\rangle\langle v,v\rangle\right).
\end{eqnarray*}
Then the desired inequality follows by replacing $a$ by $f\left(|\langle u,v \rangle|^2\right)$,\\
$b$ by $f\left(|\langle u,v\rangle |\right)f\big(\sqrt{\langle u,u\rangle\langle v,v\rangle }\big)$ and $c$ by  $f\left(\langle u,u\rangle\langle v,v\rangle\right)$ in (\ref{zzxx}).
\end{proof}
The next lemma is a simple consequence of the classical Jenson and Young's inequalities
\begin{lemma}\cite{hardy}\label{young}
     Let $a,b\geq 0$, $0\leq \mu\leq 1$ and $p,q>1$ such that $\frac{1}{p}+\frac{1}{q}=1$. Then for $r\geq 1$,
    \begin{enumerate}
        \item $a^\mu b^{1-\mu}\leq\alpha a+(1-\mu)b\leq(\mu a^r+(1-\mu)b^r)^\frac{1}{r};$
        \item $ab\leq\frac{a^p}{p}+\frac{b^q}{q}\leq\left(\frac{a^{pr}}{p}+\frac{b^{qr}}{q}\right)^\frac{1}{r}$
    \end{enumerate}
\end{lemma}
A refinement of Lemma \ref{young}$(1)$ was proven by Kittaneh et al. in $2010$ \cite{kit2010}.
\begin{equation}\label{r_0}
    a^\mu b^{1-\mu}\leq\mu a+(1-\mu)b-r_0(a^\frac{1}{2}-b^\frac{1}{2})^2
\end{equation}
where $r_0=\min\{\mu, 1-\mu\}$. Inequality \eqref{r_0} can also be written
\begin{equation}\label{youngmody}
ab+\min\left\{\frac{1}{p},\frac{1}{q}\right\}(a^\frac{p}{2}-b^\frac{q}{2})^2\leq\frac{a^p}{p}+\frac{b^q}{q},
\end{equation}
 for $a, b\geq 0$ and $p,q>1$ such that $\frac{1}{p}+\frac{1}{q}=1$.

In \cite{sheik}, Sheikhhosseini gave a new improvement of the inequality \eqref{r_0} as 
\begin{equation}
    (a^\frac{1}{p}b^\frac{1}{q})^m+(r_0)^m(a^\frac{m}{2}-b^\frac{m}{2})^2\leq\Big(\frac{a^r}{p}+\frac{b^r}{q}\Big)^\frac{m}{r},
\end{equation}
where $r_0=\min\left\{\frac{1}{p},\frac{1}{q}\right\}$, $r\geq1$ and $m\in\mathbb{N}$. In particular, if $p=q=2$, then 
\begin{equation}\label{a1/2b1/2}
  (a^\frac{1}{2}b^\frac{1}{2})^m+\frac{1}{2^m}(a^\frac{m}{2}-b^\frac{m}{2})^2\leq 2^{-\frac{m}{r}}(a^r+b^r)^\frac{m}{r}.  
\end{equation}
For $r\geq2$, a refined version of the right-hand inequality in Lemma \ref{young}$(2)$ is stated as 
\begin{lemma}\cite{gaoineq}\label{ab^r}
    If $a,b\geq 0$ and $p,q\geq 1$ are such that $\frac{1}{p}+\frac{1}{q}=1$, then for $r\geq 2$,
    \begin{equation*}
        (ab)^r\leq \left(\frac{a^p}{p}+\frac{b^q}{q}\right)^r\leq\frac{a^{pr}}{p^2}+\frac{b^{qr}}{q^2}+\frac{1}{pq}\left(a^{p(r-1)}b^q+a^pb^{q(r-1)}\right).
    \end{equation*}
    In particular, equality holds when $a^p=b^q.$
\end{lemma}

\begin{lemma}\cite{Kittaneh}\label{<Ax,x>^r}
    Let $A\in\mathscr{B}(\mathcal{H}),A\geq 0$ and $x\in\mathcal{H}$ such that $\|x\|=1$. Then 
    \begin{enumerate}
        \item $\langle Ax,x\rangle^r\leq\langle A^rx,x\rangle$ for $r\geq 1$
        \item $\langle A^rx,x\rangle\leq\langle Ax,x\rangle^r$ for  $0\leq r\leq 1$.
    \end{enumerate}
\end{lemma}

\section{\textbf{Main results}}\label{sec3}
In this section, we establish several upper bounds for the Berezin number of reproducing kernel   Hilbert space operators by employing various types of means.

\begin{theorem}\label{ber4S*T}  Let $A,B\in \mathscr{B}(\mathcal{H})$, and $\sigma_\mu,\rho_\lambda,\tau_\nu$ be any interpolational paths corresponding to the symmetric means $\sigma, \rho, \tau$ respectively, where $\mu,\nu,\lambda\in[0,1]$. If
$f$ is a multiplicative Orlicz function, then the following inequality holds for $\zeta\geq 0$
\begin{eqnarray*}
    f\left(\textbf{ber}^{4}(A^*B)\right)
    &\leq&\frac{2\zeta+1}{8\zeta+8}\left\|f\left((B^*B)^2\right)+f\left((A^*A)^2\right)\right\|^2_{\textbf{ber}}+\frac{1}{2\zeta+2}\Bigg(\bigg(f\left(\textbf{ber}^2( A^*AB^*B)\right)\sigma_\mu \\
    &&\frac{1}{2}f\left(\textbf{ber}( A^*AB^*B)\right)\left\|f\left((B^*B)^2\right)+f\left((A^*A)^2\right)\right\|_\text{ber}\bigg)\rho_\lambda~\\
    &&\bigg(f\left(\textbf{ber}^2( A^*AB^*B)\right)\tau_\nu\frac{1}{4} \left\|f\left((B^*B)^2\right)+f\left((A^*A)^2\right)\right\|_\text{ber}^2\bigg)\Bigg).
    \end{eqnarray*}
\end{theorem}
\begin{proof}
Let $\hat{k}_x$ be a normalized reproducing kernel of $\mathcal{H}$. Then we have
\begin{eqnarray}
   |\langle A^*B\hat{k}_x,\hat{k}_x\rangle|^{4}
    &\leq&\left(|\langle B^*B\hat{k}_x,\hat{k}_x\rangle||\langle A^*A\hat{k}_x,\hat{k}_x\rangle|\right)^2~~(\mbox{by Cauchy Schwarz inequality})\nonumber\\
    &=&|\langle B^*B\hat{k}_x,\hat{k}_x\rangle\langle \hat{k}_x,A^*A\hat{k}_x\rangle|^2\nonumber\\
    &\leq&\frac{2\zeta+1}{2\zeta+2}\big\langle (B^*B)^2\hat{k}_x,\hat{k}_x\big\rangle\big\langle (A^*A)^2\hat{k}_x,\hat{k}_x\big\rangle+\frac{1}{2\zeta+2}\big|\big\langle B^*B\hat{k}_x,A^*A\hat{k}_x\big\rangle\big|^2, \nonumber
\end{eqnarray} where the second inequality is derived from Lemma \ref{estimation1}, for the case $r=2$.
Applying the non-decreasing and convexity property of $f$, we get
\begin{eqnarray}
    f\left(|\langle A^*B\hat{k}_x,\hat{k}_x\rangle|^{4}\right)
    \leq\frac{2\zeta+1}{2\zeta+2}f\left(\langle (B^*B)^2\hat{k}_x,\hat{k}_x\rangle\langle (A^*A)^2\hat{k}_x,\hat{k}_x\rangle\right)+\frac{1}{2\zeta+2}f\left(|\langle B^*B\hat{k}_x,A^*A\hat{k}_x\rangle|^2\right).\nonumber\\ \label{one1}
    \end{eqnarray}
    Now, using Lemma \ref{vvvcc} together with the multiplicative property of $f$, we get
    \begin{eqnarray}
     && f\left(\langle (B^*B)^2\hat{k}_x,\hat{k}_x\rangle\langle (A^*A)^2\hat{k}_x,\hat{k}_x\rangle\right)\\
      &=&f\left(\langle (B^*B)^2\hat{k}_x,\hat{k}_x\rangle\right)f\left(\langle (A^*A)^2\hat{k}_x,\hat{k}_x\rangle\right)\nonumber\\
      &\leq&\left\langle f( (B^*B)^2)\hat{k}_x,\hat{k}_x\right\rangle\left\langle f( (A^*A)^2)\hat{k}_x,\hat{k}_x\right\rangle\nonumber\\
      &\leq&\frac{1}{4}\left\langle (f\big((B^*B)^2)+f((A^*A)^2))\hat{k}_x,\hat{k}_x\right\rangle^2~(\mbox{as}~4ab\leq (a+b)^2\,\forall \,a,b\ge 0)\nonumber\\
       &\leq&\frac{1}{4}\left\|f((B^*B)^2)+f((A^*A)^2)\right\|^2_{\textbf{ber}}. \label{one2}
    \end{eqnarray}
Again
\begin{eqnarray}
    &&f\left(|\langle B^*B\hat{k}_x,A^*A\hat{k}_x\rangle|^2\right)\nonumber\\
    &\leq&\bigg(f\left(|\langle B^*B\hat{k}_x,A^*A\hat{k}_x\rangle|^2\right)\sigma_\mu\nonumber\\
    &&f\left(|\langle B^*B\hat{k}_x,A^*A\hat{k}_x\rangle|\right)f\left(\sqrt{\langle B^*B\hat{k}_x,B^*B\hat{k}_x\rangle\langle A^*A\hat{k}_x,A^*A\hat{k}_x\rangle}\right)\bigg)~\rho_\lambda\nonumber\\
    &&\bigg(f\left(|\langle B^*B\hat{k}_x,A^*A\hat{k}_x\rangle|^2\right)\tau_\nu f\left(\langle B^*B\hat{k}_x,B^*B\hat{k}_x\rangle\langle A^*A\hat{k}_x,A^*A\hat{k}_x\rangle\right)\bigg)~~(\mbox{using Lemma} ~\ref{meanpower r})\nonumber\\
    &=&\bigg(f\left(|\langle B^*B\hat{k}_x,A^*A\hat{k}_x\rangle|^2\right)\sigma_\mu\nonumber\\
    &&f\left(|\langle B^*B\hat{k}_x,A^*A\hat{k}_x\rangle|\right)f\left(\sqrt{\langle B^*B\hat{k}_x,B^*B\hat{k}_x\rangle\langle A^*A\hat{k}_x,A^*A\hat{k}_x\rangle}\right)\bigg)~\rho_\lambda\nonumber\\
    &&\bigg(f\left(|\langle B^*B\hat{k}_x,A^*A\hat{k}_x\rangle|^2\right)\tau_\nu f\left(\langle (B^*B)^2\hat{k}_x,\hat{k}_x\rangle\right)f\left(\langle (A^*A)^2\hat{k}_x,\hat{k}_x\rangle\right)\bigg)\nonumber\\
    &\leq&\bigg(f\left(|\langle A^*AB^*B\hat{k}_x,\hat{k}_x\rangle|^2\right)~\sigma_\mu\nonumber\\
    &&f\left(|\langle A^*AB^*B\hat{k}_x,\hat{k}_x\rangle|\right)f\left(\frac{\langle (B^*B)^2\hat{k}_x,\hat{k}_x\rangle+\langle (A^*A)^2\hat{k}_x,\hat{k}_x\rangle}{2}\right)\bigg)\rho_\lambda~\nonumber\\
    &&\bigg(f\left(|\langle A^*AB^*B\hat{k}_x,\hat{k}_x\rangle|^2\right)\tau_\nu \left\langle f\left((B^*B)^2\right)\hat{k}_x,\hat{k}_x\right\rangle\left\langle f\left((A^*A)^2\right)\hat{k}_x,\hat{k}_x\right\rangle\bigg)\,\,\,\mbox{(by Lemma \ref{vvvcc})}\nonumber\\
    &\leq&\bigg(f\left(|\langle A^*AB^*B\hat{k}_x,\hat{k}_x\rangle|^2\right)\sigma_\mu\nonumber\\
    && \frac{1}{2}f\left(|\langle A^*AB^*B\hat{k}_x,\hat{k}_x\rangle|\right)\left\langle (f((B^*B)^2)+f((A^*A)^2))\hat{k}_x,\hat{k}_x\right\rangle\bigg)\rho_\lambda~\nonumber\\
    &&\bigg(f\left(|\langle A^*AB^*B\hat{k}_x,\hat{k}_x\rangle|^2\right)\tau_\nu \frac{1}{4}\left\langle (f((B^*B)^2)+f((A^*A)^2))\hat{k}_x,\hat{k}_x\right\rangle^2\bigg)~~(\mbox{as}~ 4ab\leq (a+b)^2\,\forall \,a,b\ge 0)\nonumber\\
    &\le&\bigg(f\left(\textbf{ber}^2( A^*AB^*B)\right)\sigma_\mu 
     \frac
     {1}{2}f\left(\textbf{ber}( A^*AB^*B)\right)\left\|f\left((B^*B)^2\right)+f\left((A^*A)^2\right)\right\|_\textbf{ber}\bigg)\rho_\lambda~\nonumber\\
     &&\bigg(f\left(\textbf{ber}^2( A^*AB^*B)\right)\tau_\nu\frac{1}{4} \left\|f\left((B^*B)^2\right)+f\left((A^*A)^2\right)\right\|_\textbf{ber}^2\bigg). \label{one3}
    \end{eqnarray} 
   
    From \eqref{one1}, \eqref{one2} and \eqref{one3} we have,
    \begin{eqnarray*}
     && f\left(|\langle A^*B\hat{k}_x,\hat{k}_x\rangle|^{4}\right) \\&\leq&\frac{2\zeta+1}{8\zeta+8}\left\|f\left((B^*B)^2\right)+f\left((A^*A)^2\right)\right\|^2_{\textbf{ber}}+\frac{1}{2\zeta+2}\Bigg(\bigg(f\left(\textbf{ber}^2( A^*AB^*B)\right)\sigma_\mu \\
     &&\frac{1}{2}f\left(\textbf{ber}( A^*AB^*B)\right)\left\|f\left((B^*B)^2\right)+f\left((A^*A)^2\right)\right\|_\textbf{ber}\bigg)\rho_\lambda~\\
     &&\bigg(f\left(\textbf{ber}^2( A^*AB^*B)\right)\tau_\nu\frac{1}{4} \left\|f\left((B^*B)^2\right)+f\left((A^*A)^2\right)\right\|_\textbf{ber}^2\bigg)\Bigg).
    \end{eqnarray*}
   Now, taking the supremum over all $x \in X$, we get our required result.
\end{proof}

\begin{cor}\label{saptami}
Let $A,B\in \mathscr{B}(\mathcal{H})$ and $r\geq 1$. Then for any $\mu,\nu,\lambda\in[0,1]$, the following inequality holds:
    \begin{eqnarray*}
        \textbf{ber}^{4r}(A^*B)&\leq& \frac{2\zeta+1}{8\zeta+8}\Big\|(B^*B)^{2r}+(A^*A)^{2r}\Big\|_{\textbf{ber}}^{2}\\
    &&+\frac{(1-\lambda)}{2^{\mu}(2\zeta+2)}\textbf{ber}^{2r(1-\mu)}( A^*AB^*B) \textbf{ber}^{r\mu}( A^*AB^*B)\Big\|(B^*B)^{2r}+(A^*A)^{2r})\Big\|^{\mu}_\textbf{ber}\\
    &&+\frac{\lambda}{2^{2\nu}(2\zeta+2)}\textbf{ber}^{2r(1-\nu)}( A^*AB^*B)\|(B^*B)^{2r}+(A^*A)^{2r}\|_\textbf{ber}^{2\nu}.
    \end{eqnarray*}
\end{cor}
\begin{proof}
The required result is obtained by taking $f(t)=t^r$, $t\ge 0$, and $\rho_\lambda=\nabla_\lambda, ~ \sigma_\mu=\#_\mu,~\tau_\nu=\#_\nu$ in Theorem \ref{ber4S*T}. 
\end{proof}

\begin{remark}  In  \cite[Corollary 3.7]{Taghavi}, the authors derived the following Berezin number inequality\begin{eqnarray}
    \textbf{ber}^2(A^*B)\le \frac{1}{2}\Big\|(B^*B)^2+(A^*A)^2\Big\|_{\textbf{ber}}.\label{hh}
\end{eqnarray} 
  If we take $\mathcal{H}=\mathbb{C}^2$ and $A=\begin{bmatrix}
      1& 0\\
      0& \sqrt{2}
  \end{bmatrix}$,  $B=\begin{bmatrix}
      1& 0\\
      0& \sqrt{3}
  \end{bmatrix}$, then applying inequality (\ref{hh}) to these operators gives
 the bound  $\textbf{ber}^2(A^*B)\le 6.5$
 . In contrast, for $\lambda=1,~ \zeta=0$ and $r=1$ in Corollary \ref{saptami} we obtain \[\textbf{ber}^2(A^*B)\le\bigg( \frac{169}{8} +\frac{1}{2^{2\nu+1}}36^{(1-\nu)}169^\nu\bigg)^{\frac{1}{2}}.\]
 This expression depends on a positive parameter $\nu$ and for any choice of $\nu\in[0,1],$
  we
 obtain $\bigg( \frac{169}{8} +\frac{1}{2^{2\nu+1}}36^{(1-\nu)}169^\nu\bigg)^{\frac{1}{2}}\le 6.5$. This means that Corollary \ref{saptami} provides a significant refinement
 over the estimate obtained from inequality (\ref{hh}).
\end{remark}

\begin{theorem}\label{ber2rT}
Let $A\in \mathscr{B}(\mathcal{H})$, and $\sigma_\mu,\rho_\lambda,\tau_\nu$ be any interpolational paths corresponding to the symmetric means $\sigma, \rho, \tau$ respectively, where $\mu,\nu,\lambda\in[0,1]$. If
$f$ be a Orlicz function, then the following inequality holds
    \begin{eqnarray*}
       f\left(\textbf{ber}^{2}(A)\right)  &\leq&\left(f\left(\textbf{ber}^2 (A)\right)~\sigma_\mu~  f\left(\textbf{ber} (A)\| A^*A\|_{\textbf{ber}}^\frac{1}{2}\right)\right)\rho_\lambda\\
   &&\left(f\left(\textbf{ber}^2 (A)\right)~ \tau_\nu ~\frac{1}{2} \left\|f(|A|^{2})+f(|A^*|^{2})\right\|_\textbf{ber}\right).
   \end{eqnarray*}
\end{theorem}
\begin{proof}
    Let $\hat{k}_x$ be a normalized reproducing kernel of $\mathcal{H}$. Then we have
    \begin{eqnarray*}
       |\langle A\hat{k}_x,\hat{k}_x\rangle|^{2}&\leq&|\langle A\hat{k}_x,\hat{k}_x\rangle|\|A\hat{k}_x\|\|\hat{k}_x\|\\
         &=&|\langle A\hat{k}_x,\hat{k}_x\rangle|\langle A\hat{k}_x,A\hat{k}_x\rangle^\frac{1}{2}~~(\mbox{as}~ \|\hat{k}_x\|=1)\\
        &=&|\langle A^*\hat{k}_x,\hat{k}_x\rangle|\langle A\hat{k}_x,A\hat{k}_x\rangle^\frac{1}{2}~~(\mbox{as}~\langle A\hat{k}_x,\hat{k}_x\rangle=\langle \hat{k}_x,A^*\hat{k}_x\rangle)\\
        &\leq&\langle A^*\hat{k}_x,A^*\hat{k}_x\rangle^\frac{1}{2}\langle A\hat{k}_x,A\hat{k}_x\rangle^\frac{1}{2}~~(\mbox{by the Cauchy-Schwarz inequality})\\
        &=&\langle AA^*\hat{k}_x,\hat{k}_x\rangle^\frac{1}{2}\langle A^*A\hat{k}_x,\hat{k}_x\rangle^\frac{1}{2}\\
        &\leq&\langle |A^*|^{2}\hat{k}_x,\hat{k}_x\rangle^\frac{1}{2}\langle |A|^{2}\hat{k}_x,\hat{k}_x\rangle^\frac{1}{2}\\
        &\leq& \frac{\langle |A^*|^{2}\hat{k}_x,\hat{k}_x\rangle+\langle |A|^{2}\hat{k}_x,\hat{k}_x\rangle}{2}\\
        &&(\mbox{using the arithmetic-geometric mean inequality})
    \end{eqnarray*}
    that is,
    \begin{equation*}
        |\langle A\hat{k}_x,\hat{k}_x\rangle|^{2}\leq|\langle A\hat{k}_x,\hat{k}_x\rangle|\langle A\hat{k}_x,A\hat{k}_x\rangle^\frac{1}{2}\leq\frac{\langle |A^*|^{2}\hat{k}_x,\hat{k}_x\rangle+\langle |A|^{2}\hat{k}_x,\hat{k}_x\rangle}{2}.
    \end{equation*}
  Using the non-decreasing and convexity property of $f$, we get 
  \begin{eqnarray*}
      f\left(|\langle A\hat{k}_x,\hat{k}_x\rangle|^{2}\right) &\le& f\left(|\langle A\hat{k}_x,\hat{k}_x\rangle|\langle A\hat{k}_x,A\hat{k}_x\rangle^\frac{1}{2}\right)\\
      & \le& \frac{1}{2} \left(f(\langle |A^*|^{2}\hat{k}_x,\hat{k}_x\rangle)+f(\langle |A|^{2}\hat{k}_x,\hat{k}_x\rangle)\right)\\
       & \le& \frac{1}{2} \left(\langle f(|A^*|^{2})\hat{k}_x,\hat{k}_x\rangle+\langle f(|A|^{2})\hat{k}_x,\hat{k}_x\rangle\right) \,\mbox{(by Lemma \ref{vvvcc})}\\
        &=& \frac{1}{2} \left(\langle (f(|A|^{2})+f(|A^*|^{2}))\hat{k}_x,\hat{k}_x\rangle\right). 
  \end{eqnarray*}
   % Now, by replacing $a$ by $f\left(|\langle A\hat{k}_x,\hat{k}_x\rangle|^{2}\right)$, $b$ by $f\left(|\langle A\hat{k}_x,\hat{k}_x\rangle|\langle A\hat{k}_x,A\hat{k}_x\rangle^\frac{1}{2}\right)$ and $c$ by $ \frac{1}{2} \left(\langle (f(|A|^{2})+f(|A^*|^{2}))\hat{k}_x,\hat{k}_x\rangle\right) $ in 
   Using inequality \eqref{zzxx}, we get,
\begin{eqnarray*}
   f\left(|\langle A\hat{k}_x,\hat{k}_x\rangle|^{2}\right)&\leq&\left(f\left(|\langle A\hat{k}_x,\hat{k}_x\rangle|^{2}\right)~\sigma_\mu~  f\left(|\langle A\hat{k}_x,\hat{k}_x\rangle|\langle A\hat{k}_x,A\hat{k}_x\rangle^\frac{1}{2}\right)\right)\\
   &&~\rho_\lambda~  \left(f\left(|\langle A\hat{k}_x,\hat{k}_x\rangle|^{2}\right)~ \tau_\nu ~\frac{1}{2} \langle (f(|A|^{2})+f(|A^*|^{2}))\hat{k}_x,\hat{k}_x\rangle \right)\\
  &\leq&\left(f\left(\textbf{ber}^2 (A)\right)~\sigma_\mu~  f\left(\textbf{ber} (A)\| A^*A\|_{\textbf{ber}}^\frac{1}{2}\right)\right)\\
   &&~\rho_\lambda~  \left(f\left(\textbf{ber}^2 (A)\right)~ \tau_\nu ~\frac{1}{2} \left\|f(|A|^{2})+f(|A^*|^{2})\right\|_\textbf{ber}\right).\\
\end{eqnarray*}
Taking the supremum over all $x \in X$, we obtain our required result.
   \end{proof}
   
\begin{cor}\label{astomi}
Let $A\in \mathscr{B}(\mathcal{H})$ and $r\geq 1$. Then for any $\mu,\nu,\lambda\in[0,1]$, the following inequality holds
       \begin{eqnarray*}
       \textbf{ber}^{2r}(A) &\leq& (1-\lambda)\left(\textbf{ber}^{2r(1-\mu)} (A)\textbf{ber}^{r\mu} (A)\| A^*A\|_{\textbf{ber}}^\frac{r\mu}{2}\right)\\
   &&+\frac{\lambda}{2^{\nu}}  \left(\textbf{ber}^{2r(1-\nu)} (A) \left\||A|^{2r}+|A^*|^{2r}\right\|^{\nu}_\textbf{ber}\right).\\
   \end{eqnarray*}
\end{cor}
\begin{proof}
The required result is obtained by taking $f(t)=t^r$, $t\ge 0$, and $\rho_\lambda=\nabla_\lambda, ~ \sigma_\mu=\#_\mu,~\tau_\nu=\#_\nu$ in Theorem \ref{ber2rT}. 
\end{proof}
\begin{remark} From \cite[Corollary 3.5(i)]{Taghavi}, we have
\begin{eqnarray}
    \textbf{ber}^{2}(A)\leq\frac{1}{2}\left\||A^*A+AA^*\right\|_\textbf{ber}.\label{jj1}
\end{eqnarray}
  For $\mathcal{H}=\mathbb{C}^2$ and $A=\begin{bmatrix}
      1& 1\\
      0& 0
  \end{bmatrix}$, inequality (\ref{jj1}) yields the  bound  $\textbf{ber}(A)\le\sqrt{ 1.5}$. In comparison, choosing $\lambda=0$, $\mu=1$ and $r=1$ in Corollary \ref{astomi} we obtain \[\textbf{ber}(A)\le 1.\]

Hence, Corollary \ref{astomi} provides an improvement over the bound obtained from inequality (\ref{jj1}).
  \end{remark}

\begin{theorem}\label{producf}Let $A,B,C,D,M,N\in \mathscr{B}(\mathcal{H})$, and $\sigma_\mu,\tau_\nu$ be any interpolational paths corresponding to the symmetric means $\sigma, \tau$, respectively, where $\mu,\nu\in[0,1]$. If
$f$ is a multiplicative Orlicz function and $\alpha,\beta\in[0,1]$, then the following inequality holds

    \begin{eqnarray*}
      &&f\left(\textbf{ber}\left(\frac{A^*MB+C^*ND}{2}\right)\right)\\
      &\le &\frac{1}{2}\left(f\left(\textbf{ber}( A^*MB)\right)\sigma_\mu\frac{1}{2}\left\|f(B^*|M|^{2\alpha}B)+f(A^*|M^*|^{2(1-\alpha)}A)\right\|_{\textbf{ber}}\right)\\
       &&+\frac{1}{2}\left(f(\textbf{ber}(C^*ND))\tau_\nu\frac{1}{2}\left\|f(D^*|N|^{2\beta}D)+f(C^*|N^*|^{2(1-\beta)}C)\right\|_{\textbf{ber}}\right).
    \end{eqnarray*}
\end{theorem}
\begin{proof}
Let $\hat{k}_x$ be a normalized reproducing kernel of $\mathcal{H}$. Then for convexity property of $f$, we have
\begin{equation}
f\left(\left|\left\langle\left(\frac{A^*MB+C^*ND}{2}\right)\hat{k}_x,\hat{k}_x\right\rangle\right|\right)
\leq\frac{1}{2}f\left(|\langle MB\hat{k}_x,A\hat{k}_x\rangle|\right)+\frac{1}{2}f\left(|\langle ND\hat{k}_x,C\hat{k}_x\rangle|\right).\label{Two1}   
\end{equation}
Now by Lemma \ref{txy^2}, we have for $0\leq\alpha\leq 1$, 
\begin{equation*}
  f\left(|\langle MB\hat{k}_x,A\hat{k}_x\rangle|\right)\\
 \leq f\left(\langle |M|^{2\alpha}B\hat{k}_x,B\hat{k}_x\rangle^{\frac{1}{2}}\langle |M^*|^{2(1-\alpha)}A\hat{k}_x,A\hat{k}_x\rangle^{\frac{1}{2}}\right).  
\end{equation*}
Since $a \leq a \sigma_{\mu} b \leq b$ for $0 \leq a \leq b$, choosing
\[a=f\left(|\langle MB\hat{k}_x,A\hat{k}_x\rangle|\right) \,\mbox{ and } \,b= f\left(\langle |M|^{2\alpha}B\hat{k}_x,B\hat{k}_x\rangle^{\frac{1}{2}}\langle |M^*|^{2(1-\alpha)}A\hat{k}_x,A\hat{k}_x\rangle^{\frac{1}{2}}\right),\] and using the monotone increasing property of $\sigma_\mu$, we deduce
\begin{eqnarray}
 &&f\left(|\langle MB\hat{k}_x,A\hat{k}_x\rangle|\right)\nonumber\\
 &\leq&\left(f\left(|\langle MB\hat{k}_x,A\hat{k}_x\rangle|\right)\sigma_\mu f\left(\langle |M|^{2\alpha}B\hat{k}_x,B\hat{k}_x\rangle^{\frac{1}{2}}\langle |M^*|^{2(1-\alpha)}A\hat{k}_x,A\hat{k}_x\rangle^{\frac{1}{2}}\right)\right)\nonumber\\
 &\leq&f\left(|\langle A^*MB\hat{k}_x,\hat{k}_x\rangle|\right)\sigma_\mu f\left(\frac{\langle |M|^{2\alpha}B\hat{k}_x,B\hat{k}_x\rangle+\langle |M^*|^{2(1-\alpha)}A\hat{k}_x,A\hat{k}_x\rangle}{2}\right)\nonumber\\
 &\leq&f\left(\textbf{ber}( A^*MB)\right)\sigma_\mu\frac{1}{2}\left(f(\langle B^*|M|^{2\alpha}B\hat{k}_x,\hat{k}_x\rangle)+f(\langle A^*|M^*|^{2(1-\alpha)}A\hat{k}_x,\hat{k}_x\rangle)\right)\nonumber\\
 &\leq&f\left(\textbf{ber}( A^*MB)\right)\sigma_\mu\frac{1}{2}\left(\langle(f(B^*|M|^{2\alpha}B)+f(A^*|M^*|^{2(1-\alpha)}A)\hat{k}_x,\hat{k}_x\rangle\right) ~~(\mbox{using Lemma \ref{vvvcc}})\nonumber\\
 &\leq&f\left(\textbf{ber}( A^*MB)\right)\sigma_\mu\frac{1}{2}\left\|f\left(B^*|M|^{2\alpha}B\right)+f\left(A^*|M^*|^{2(1-\alpha)}A\right)\right\|_{\textbf{ber}}\label{Two2}.
\end{eqnarray}
Similarly, for $0\leq \beta\leq1$, it follows 
\begin{equation}\label{Two3}
  f\left(|\langle ND\hat{k}_x,C\hat{k}_x\rangle|\right)\leq f\left(\textbf{ber}( C^*ND)\right)\tau_\nu\frac{1}{2}\left\|f\left(D^*|N|^{2\beta}D\right)+f\left(C^*|N^*|^{2(1-\beta)}C\right)\right\|_{\textbf{ber}}.  
\end{equation}
Therefore from equations \eqref{Two1}, \eqref{Two2} and \eqref{Two3}, we have
\begin{eqnarray*}
      &&f\left(\left|\left\langle\left(\frac{A^*MB+C^*ND}{2}\right)\hat{k}_x,\hat{k}_x\right\rangle\right|\right)\\
      &\le &\frac{1}{2}\left(f\left(\textbf{ber}( A^*MB)\right)\sigma_\mu\frac{1}{2}\left\|f(B^*|M|^{2\alpha}B)+f(A^*|M^*|^{2(1-\alpha)}A)\right\|_{\textbf{ber}}\right)\\
       &&+\frac{1}{2}\left(f(\textbf{ber}(C^*ND))\tau_\nu\frac{1}{2}\left\|f(D^*|N|^{2\beta}D)+f(C^*|N^*|^{2(1-\beta)}C)\right\|_{\textbf{ber}}\right).
    \end{eqnarray*}
Now, taking the supremum over all $x \in X$, we get our required result.
\end{proof}

\begin{cor}\label{nabami}
Let $A,B,C,D,M,N\in \mathscr{B}(\mathcal{H})$. Then for any $\alpha,\beta,\mu,\nu\in[0,1]$, the following inequality holds
    \begin{eqnarray*}
    &&\textbf{ber}^r\left(\frac{A^*MB+C^*ND}{2}\right)\\
    &\le &\frac{1}{2^{\mu+1}}\left(\textbf{ber}^{r(1-\mu)}( A^*MB)\left\|(B^*|M|^{2\alpha}B)^r+(A^*|M^*|^{2(1-\alpha)}A)^r\right\|^{\mu}_{\textbf{ber}}\right)\\
       &&+\frac{1}{2^{\nu+1}}\left(\textbf{ber}^{r(1-\nu)}(  C^*ND)\left\|(D^*|N|^{2\beta}D)^r+(C^*|N^*|^{2(1-\beta)}C)^r\right\|^{\nu}_{\textbf{ber}}\right).
    \end{eqnarray*}
\end{cor}
\begin{proof}
The required result is obtained by taking $f(t)=t^r$, $t\ge 0$, and $ \sigma_\mu=\#_\mu,~\tau_\nu=\#_\nu$ in Theorem \ref{producf}. 
\end{proof}

\begin{remark} 
In \cite[Theorem 2.5(ii)]{hajmohamadi2020improvements} it is obtained that 
\begin{equation}
    \textbf{ber}(A^*MB)\leq \frac{1}{2}\left\|B^*|M|^{2\alpha}B+A^*|M^*|^{2(1-\alpha)}A\right\|_{\textbf{ber}}.\label{hh3}
\end{equation}
 In contrast,
for $A=C$, $B=D$, $M=N$, $\mu=\nu$ and $r=1$,  Corollary \ref{nabami} becomes
\begin{eqnarray}
       \textbf{ber}(A^*MB)\leq \frac{1}{2^{\mu}}\textbf{ber}^{(1-\mu)}( A^*MB)\left\|B^*|M|^{2\alpha}B+A^*|M^*|^{2(1-\alpha)}A\right\|^{\mu}_{\textbf{ber}}.\label{hh4}
\end{eqnarray}
It follows  that for any  choice of $\mu\in[0,1]$,
 inequality
(\ref{hh4}) provides a sharper bound than inequality (\ref{hh3}).
\end{remark}

We note the scalar inequality
\begin{eqnarray}\label{a+ib}
    |a+b|\le \sqrt{2}|a+ib|\,\,\,\mbox{where $a,b\in \mathbb{R}$}.
\end{eqnarray}

Recall that the interpolation path of the arithmetic mean is given by $a\nabla_\lambda b=(1-\lambda)a+\lambda b,~0\leq\lambda\leq 1$ for $0 \leq a \leq b$. Since $a \leq a \nabla_\lambda b \leq b$,  by taking
\[a=|\langle Au, v\rangle|^2 \,\mbox{ and } \,b= \langle|A|^{2\alpha}u,u\rangle\langle|A^*|^{2(1-\alpha)}v,v\rangle ,\]
 we obtain
\begin{equation}\label{meanwithalpha}
|\langle Au, v\rangle|^2\leq |\langle Au, v\rangle|^2 \nabla_\lambda\langle|A|^{2\alpha}u,u\rangle\langle|A^*|^{2(1-\alpha)}v,v\rangle, 
\end{equation}  which follows from Lemma \ref{txy^2}. The two inequalities mentioned above are used to derive the following result.

\begin{theorem}\label{ber2a+b}
Let $A,B\in \mathscr{B}(\mathcal{H})$ and $f_1$, $f_2$ be the complementary Orlicz functions corresponding to $g_1$, $g_2$, respectively. Then for any $\alpha,\beta,\lambda\in[0,1]$, the following inequality holds
    \begin{eqnarray*}
        \textbf{ber}^2(A+B)
        &\le&2 \Big(\left(\textbf{ber}^2 (A)+\textbf{ber}^2 (B)\right)\nabla_{\lambda}\\
        &&\sqrt{2}\,\textbf{ber} \left(f_1(|A|^{2\alpha})+f_2( |B|^{2\beta})+i\left(g_1(|A^*|^{2(1-\alpha)})+g_2(|B^*|^{2(1-\beta)})\right)\right)\Big).
    \end{eqnarray*}
\end{theorem}

\begin{proof} Let $\hat{k}_x$ be a normalized reproducing kernel of $\mathcal{H}$. Then we have
    \begin{eqnarray*}
        &&|\langle(A+B)\hat{k}_x,\hat{k}_x\rangle|^2\\
        &\le& \left(|\langle A\hat{k}_x,\hat{k}_x\rangle|+|\langle B\hat{k}_x,\hat{k}_x\rangle|\right)^2\\
        &\le&2 \left(|\langle A\hat{k}_x,\hat{k}_x\rangle|^2+|\langle B\hat{k}_x,\hat{k}_x\rangle|^2\right)\\
        &\le&2 \Big(|\langle A\hat{k}_x,\hat{k}_x\rangle|^2\nabla_{\lambda}\langle |A|^{2\alpha}\hat{k}_x,\hat{k}_x\rangle\langle |A^*|^{2(1-\alpha)}\hat{k}_x,\hat{k}_x\rangle\\
        &&+\big|\langle B\hat{k}_x,\hat{k}_x\rangle\big|^2\nabla_{\lambda}\langle |B|^{2\beta}\hat{k}_x,\hat{k}_x\rangle\langle |B^*|^{2(1-\beta)}\hat{k}_x,\hat{k}_x\rangle\Big)~~(\mbox{using \eqref{meanwithalpha}})\\
        &=&2\bigg(\left(|\langle A\hat{k}_x,\hat{k}_x\rangle|^2+|\langle B\hat{k}_x,\hat{k}_x\rangle|^2\right)\nabla_{\lambda}\Big(\langle |A|^{2\alpha}\hat{k}_x,\hat{k}_x\rangle\langle |A^*|^{2(1-\alpha)}\hat{k}_x,\hat{k}_x\rangle\\
        &&+\langle |B|^{2\beta}\hat{k}_x,\hat{k}_x\rangle\langle |B^*|^{2(1-\beta)}\hat{k}_x,\hat{k}_x\rangle\Big)\bigg)\\
        &\le&2 \bigg(\Big(\textbf{ber}^2 (A)+\textbf{ber}^2 (B)\Big)\nabla_{\lambda}\Big(f_1(\langle |A|^{2\alpha}\hat{k}_x,\hat{k}_x\rangle)+g_1(\langle |A^*|^{2(1-\alpha)}\hat{k}_x,\hat{k}_x\rangle)\\
        &&+f_2(\langle |B|^{2\beta}\hat{k}_x,\hat{k}_x\rangle)+g_2(\langle |B^*|^{2(1-\beta)}\hat{k}_x,\hat{k}_x\rangle)\Big)\bigg)~~(\mbox{using Lemma \ref{orlizineq}})\\
        &\le& 2\bigg(\Big(\textbf{ber}^2 (A)+\textbf{ber}^2 (B)\Big)\nabla_{\lambda}\Big(\langle f_1(|A|^{2\alpha})\hat{k}_x,\hat{k}_x\rangle+\langle f_2( |B|^{2\beta})\hat{k}_x,\hat{k}_x\rangle\\
        &&+\langle g_1(|A^*|^{2(1-\alpha)})\hat{k}_x,\hat{k}_x\rangle+
        \langle   g_2(|B^*|^{2(1-\beta)})\hat{k}_x,\hat{k}_x\rangle\Big)\Big)~~(\mbox{using Lemma \ref{vvvcc}})\\
        &\le&2 \bigg(\Big(\textbf{ber}^2 (A)+\textbf{ber}^2 (B)\Big)\nabla_{\lambda}\sqrt{2}\Big|\big\langle (f_1(|A|^{2\alpha})+f_2(|B|^{2\beta})\\
        &&+i(g_1(|A^*|^{2(1-\alpha)})+g_2(|B^*|^{2(1-\beta)})))\hat{k}_x,\hat{k}_x\big\rangle\Big|\bigg)~(\mbox{using \eqref{a+ib}})\\
&\le&2 \Big(\left(\textbf{ber}^2 (A)+\textbf{ber}^2 (B)\right)\nabla_{\lambda}\\
&&\sqrt{2}\,\textbf{ber} \left(f_1(|A|^{2\alpha})+f_2( |B|^{2\beta})+i(g_1(|A^*|^{2(1-\alpha)})+g_2(|B^*|^{2(1-\beta)})\big)\right)\Big).\end{eqnarray*}
Now, taking the supremum over all $x \in X$, we get our required result.
\end{proof}

Taking $f_1(t)=g_1(t)=f_2(t)=g_2(t)=\frac{t^2}{2}$ and $\alpha=\beta=\frac{1}{2}$ in Theorem \ref{ber2a+b}, we get the following result.
\begin{cor}\label{dasomi}
    Let $A,B\in \mathscr{B}(\mathcal{H})$. Then for any $\lambda\in[0,1]$, the following inequality holds
    \begin{eqnarray*}
        &&\textbf{ber}^2(A+B)\\
        &\le&2(1-\lambda)\left(\textbf{ber}^2 (A)+\textbf{ber}^2 (B)\right)+\sqrt{2}\lambda\,\textbf{ber} \left(|A|^{2}+|B|^{2}+i(|A^*|^{2}+|B^*|^{2})\right).
    \end{eqnarray*}
\end{cor}

\begin{remark}
     From \cite[Corollary 2.19(ii)]{pintuacta}, we have 
\begin{eqnarray}
   \textbf{ber}^2(A+B)\le\sqrt{2}\,\textbf{ber} \left(|A|^{2}+|B|^{2}+i(|A^*|^{2}+|B^*|^{2})\right).\label{pp12}
\end{eqnarray}
  For for $\lambda=1$, $\mathcal{H}=\mathbb{C}^2$ and the operators $A=\begin{bmatrix}
      1& 0\\
      1& 0
  \end{bmatrix}$, $B=\begin{bmatrix}
      2& 2\\
      0& 0
  \end{bmatrix}$, the inequality (\ref{pp12}) yields the  bound  $\textbf{ber}^2(A+B)\le\sqrt{234}$ whereas Corollary \ref{dasomi} gives
  \[\textbf{ber}^2(A+B)\le (1-\lambda)10+ \lambda\sqrt{ 234}.\]
 This expression depends on a positive parameter $\lambda$ and for any choice of $\lambda\in[0,1],$
  we
 obtain $ (1-\lambda)10+ \lambda\sqrt{ 234}\le \sqrt{ 234}$. This means that Corollary \ref{dasomi} provides a significant refinement
 over the estimate obtained from inequality (\ref{pp12}).
\end{remark}

\section{\textbf{Some improvements for Berezin number using refined Young's inequalities}}\label{sec4}
 In this section, we derive several refined Berezin number inequalities for bounded linear operators on reproducing kernel Hilbert spaces using refined Young's inequalities.  
\begin{theorem}\label{cow*}
 Let $A\in \mathscr{B}(\mathcal{H})$ and $p,q\geq 1$ such that $\frac{1}{p}+\frac{1}{q}=1$. Then, for $r\geq2$ and $\zeta\geq 0$, the following inequality holds:
\begin{eqnarray*}
        \textbf{ber}^{2r}(A)&\leq&\frac{2\zeta+1}{2\zeta+2}\bigg(\frac{\||A|^{pr}\|_\textbf{ber}}{p^2}+\frac{\||A^*|^{qr}\|_\textbf{ber}}{q^2}\\
        &&+\frac{\||A|^{2p(r-1)}\|_\textbf{ber}^\frac{1}{2}\||A^*|^{2q}\|_\textbf{ber}^\frac{1}{2}+\||A|^{2p}\|_\textbf{ber}^\frac{1}{2}\||A^*|^{2q(r-1)}\|_\textbf{ber}^\frac{1}{2}}{pq}\bigg)+\frac{\textbf{ber}^{r}(A^2)}{2\zeta+2}.
    \end{eqnarray*}
\end{theorem}
\begin{proof}
Let $\hat{k}_x$ be a normalized reproducing kernel of $\mathcal{H}$. Then 
    \begin{eqnarray*}
        &&|\langle A\hat{k}_x,\hat{k}_x\rangle|^{2r}\\
        &=&\left(|\langle A\hat{k}_x,\hat{k}_x\rangle||\langle \hat{k}_x,A^*\hat{k}_x\rangle|\right)^r\\
        &\leq&\frac{2\zeta+1}{2\zeta+2}\langle A\hat{k}_x,A\hat{k}_x\rangle^\frac{r}{2}\langle A^*\hat{k}_x,A^*\hat{k}_x\rangle^\frac{r}{2}+\frac{1}{2\zeta+2}|\langle A\hat{k}_x,A^*\hat{k}_x\rangle|^r~~(\mbox{using Lemma \ref{estimation1}})\\
        &=&\frac{2\zeta+1}{2\zeta+2}\left(\langle |A|^2\hat{k}_x,\hat{k}_x\rangle^\frac{1}{2}\langle |A^*|^2\hat{k}_x,\hat{k}_x\rangle^\frac{1}{2}\right)^r+\frac{1}{2\zeta+2}\left(|\langle A^2\hat{k}_x,\hat{k}_x\rangle|\right)^r\\
       &\leq&\frac{2\zeta+1}{2\zeta+2}\bigg(\frac{\langle |A|^2\hat{k}_x,\hat{k}_x\rangle^\frac{pr}{2}}{p^2}+\frac{\langle |A^*|^2\hat{k}_x,\hat{k}_x\rangle^\frac{qr}{2}}{q^2}\\
       &&+\frac{\langle |A|^2\hat{k}_x,\hat{k}_x\rangle^\frac{p(r-1)}{2}\langle |A^*|^2\hat{k}_x,\hat{k}_x\rangle^\frac{q}{2}+\langle |A|^2\hat{k}_x,\hat{k}_x\rangle^\frac{p}{2}\langle |A^*|^2\hat{k}_x,\hat{k}_x\rangle^\frac{q(r-1)}{2}}{pq}\bigg)\\
       &&+\frac{1}{2\zeta+2}(|\langle A^2\hat{k}_x,\hat{k}_x\rangle|^{r})~~(\mbox{using Lemma \ref{ab^r}})\\
       &\leq&\frac{2\zeta+1}{2\zeta+2}\bigg(\frac{\langle |A|^{pr}\hat{k}_x,\hat{k}_x\rangle}{p^2}+\frac{\langle |A^*|^{qr}\hat{k}_x,\hat{k}_x\rangle}{q^2}\\
       &&+\frac{1}{pq}\left(\langle |A|^{2p(r-1)}\hat{k}_x,\hat{k}_x\rangle^\frac{1}{2}\langle |A^*|^{2q}\hat{k}_x,\hat{k}_x\rangle^\frac{1}{2}+\langle |A|^{2p}\hat{k}_x,\hat{k}_x\rangle^\frac{1}{2}\langle |A^*|^{2q(r-1)}\hat{k}_x,\hat{k}_x\rangle^\frac{1}{2}\right)\bigg)\\
       &&+\frac{|\langle A^2\hat{k}_x,\hat{k}_x\rangle|^{r}}{2\zeta+2}~~(\mbox{using Lemma \ref{<Ax,x>^r}})\\
       &\leq&\frac{2\zeta+1}{2\zeta+2}\bigg(\frac{\||A|^{pr}\|_\textbf{ber}}{p^2}+\frac{\||A^*|^{qr}\|_\textbf{ber}}{q^2}+\frac{\||A|^{2p(r-1)}\|_\textbf{ber}^\frac{1}{2}\||A^*|^{2q}\|_\textbf{ber}^\frac{1}{2}+\||A|^{2p}\|_\textbf{ber}^\frac{1}{2}\||A^*|^{2q(r-1)}\|_\textbf{ber}^\frac{1}{2}}{pq}\bigg)\\
       &&+\frac{1}{2\zeta+2}\textbf{ber}^{r}(A^2).
    \end{eqnarray*}
Taking the supremum over $x\in X$, we get our required result.
\end{proof}

Taking $p=q=2$ in Theorem \ref{cow*}, we obtain the following result.
\begin{cor}\label{corcow1}
   Let $A\in \mathscr{B}(\mathcal{H})$. Then, 
    \begin{eqnarray*}
        \textbf{ber}^{2r}(A)&\leq&\frac{2\zeta+1}{2\zeta+2}\bigg(\frac{\||A|^{2r}\|_\textbf{ber}}{4}+\frac{\||A^*|^{2r}\|_\textbf{ber}}{4}\\
        &&+\frac{\||A|^{4(r-1)}\|_\textbf{ber}^\frac{1}{2}\||A^*|^{4}\|_\textbf{ber}^\frac{1}{2}+\||A|^{4}\|_\textbf{ber}^\frac{1}{2}\||A^*|^{4(r-1)}\|_\textbf{ber}^\frac{1}{2}}{4}\bigg)+\frac{1}{2\zeta+2}\textbf{ber}^{r}(A^2),
    \end{eqnarray*}
where $r\geq2$ and $\zeta\geq 0$.
\end{cor}

\begin{remark}
According to \cite[Theorem 2.5]{esti}, it follows that 
\begin{equation}\label{cow12*} 
      \textbf{ber}^{2s}(A)\leq \frac{2\zeta+1}{2\zeta+2}\left\|\frac{|A|^{2s}}{2}+\frac{|A^*|^{2s}}{2}\right\|_\textbf{ber}+\frac{1}{2\zeta+2}\textbf{ber}^s(A^2),
\end{equation} where $s\ge 1$.
 Now, we consider the matrix $A=\begin{pmatrix}
     1 & 1\\
     0 & 0
 \end{pmatrix}$ and let
$S=2$. Applying inequality \eqref{cow12*}, we obtain $$\textbf{ber}^4(A)\le \frac{2\zeta+1}{2\zeta+2}6 +\frac{1}{2\zeta+2}.$$ 
Moreover, when $r=2$, Corollary \ref{corcow1} yields 
 $$\textbf{ber}^4(A)\le \frac{2\zeta+1}{2\zeta+2}2.914 +\frac{1}{2\zeta+2}.$$
Therefore, Corollary \ref{corcow1} provides a sharper bound than that obtained in \cite[Theorem 2.5]{esti}.
\end{remark}

\begin{theorem}\label{temp}
 Let  $A\in\mathscr{B}(\mathcal{H})$ and $p,q\geq 1$ such that $\frac{1}{p}+\frac{1}{q}=1$. Then, for $r\geq2$,
 \begin{eqnarray*}
   \textbf{ber}^r(A)\leq\left\|\frac{|A|^\frac{pr}{2}}{p}+\frac{|A^*|^\frac{qr}{2}}{q}\right\|_\textbf{ber}-\min\left\{\frac{1}{p},\frac{1}{q}\right\}\inf_{x\in X}\left(\langle|A|\hat{k}_x,\hat{k}_x\rangle^\frac{pr}{4}-\langle|A^*|\hat{k}_x,\hat{k}_x\rangle^\frac{qr}{4}\right)^2.
 \end{eqnarray*}
\end{theorem}
\begin{proof} Let $\hat{k}_x$ be a normalized reproducing kernel of $\mathcal{H}$. Then, from Lemma \ref{txy^2} with $\alpha=\frac{1}{2}$, we have
   \begin{eqnarray*}
          |\langle A\hat{k}_x, \hat{k}_x\rangle|^r&\leq& \langle|A|\hat{k}_x,\hat{k}_x\rangle^\frac{r}{2}\langle|A^*|\hat{k}_x,\hat{k}_x\rangle^\frac{r}{2}\\ 
          &\leq&\frac{1}{p}\langle|A|\hat{k}_x,\hat{k}_x\rangle^\frac{pr}{2}+\frac{1}{q}\langle|A^*|\hat{k}_x,\hat{k}_x\rangle^\frac{qr}{2}-\min\left\{\frac{1}{p},\frac{1}{q}\right\}\left(\langle|A|\hat{k}_x,\hat{k}_x\rangle^\frac{pr}{4}-\langle|A^*|\hat{k}_x,\hat{k}_x\rangle^\frac{qr}{4}\right)^2\\
          &&(\mbox{using \eqref{youngmody}})\\
         &\leq&\frac{1}{p}\langle|A|^\frac{pr}{2}\hat{k}_x,\hat{k}_x\rangle+\frac{1}{q}\langle|A^*|^\frac{qr}{2}\hat{k}_x,\hat{k}_x\rangle-\min\left\{\frac{1}{p},\frac{1}{q}\right\}\left(\langle|A|\hat{k}_x,\hat{k}_x\rangle^\frac{pr}{4}-\langle|A^*|\hat{k}_x,\hat{k}_x\rangle^\frac{qr}{4}\right)^2\\
         &&(\mbox{using Lemma \ref{<Ax,x>^r}})\\
         &=&\left\langle\left(\frac{|A|^\frac{pr}{2}}{p}+\frac{|A^*|^\frac{qr}{2}}{q}\right)\hat{k}_x,\hat{k}_x\right\rangle-\min\left\{\frac{1}{p},\frac{1}{q}\right\}\left(\langle|A|\hat{k}_x,\hat{k}_x\rangle^\frac{pr}{4}-\langle|A^*|\hat{k}_x,\hat{k}_x\rangle^\frac{qr}{4}\right)^2\\
         &\leq&\left\|\frac{|A|^\frac{pr}{2}}{p}+\frac{|A^*|^\frac{qr}{2}}{q}\right\|_\textbf{ber}-\min\left\{\frac{1}{p},\frac{1}{q}\right\}\inf_{x\in X}\left(\langle|A|\hat{k}_x,\hat{k}_x\rangle^\frac{pr}{4}-\langle|A^*|\hat{k}_x,\hat{k}_x\rangle^\frac{qr}{4}\right)^2.
    \end{eqnarray*}
     Taking the supremum over $x\in X$, we get our desired result.
\end{proof}

 The following corollary is derived by setting $p=q=2$ in Theorem \ref{temp}.
\begin{cor}\label{infcor} 
 Let  $A\in\mathscr{B}(\mathcal{H})$. Then, for $r\geq2$,
    \begin{eqnarray*}
   \textbf{ber}^r(A)\leq\frac{1}{2}\left\||A|^r+|A^*|^r\right\|_\textbf{ber}-\frac{1}{2}\inf_{x\in X}\left(\langle|A|\hat{k}_x,\hat{k}_x\rangle^\frac{r}{2}-\langle|A^*|\hat{k}_x,\hat{k}_x\rangle^\frac{r}{2}\right)^2.
 \end{eqnarray*}
\end{cor}
\begin{remark}
  In \cite[Corollary $3.5$]{Taghavi} it is given that, for $s\geq1$,
  $$\textbf{ber}^s(A)\leq\frac{1}{2}\textbf{ber}\left(|A|^s+|A^*|^s\right).$$ 
  Since $\inf_{x\in X}\left(\langle|A|\hat{k}_x,\hat{k}_x\rangle^\frac{r}{2}-\langle|A^*|\hat{k}_x,\hat{k}_x\rangle^\frac{r}{2}\right)^2\geq 0$, so Corollary \ref{infcor} gives better bound for $r\geq 2$.
\end{remark}

\begin{theorem}\label{goat}
 Let $A,B\in \mathscr{B}(\mathcal{H})$. Then 
\begin{eqnarray*}
    \textbf{ber}^{2r}(A^*B)&\leq&\frac{2\zeta+1}{2\zeta+2}\left\|\frac{|B|^{4m}+|A|^{4m}}{2}\right\|_\textbf{ber}^\frac{r}{m}-\frac{2\zeta+1}{2^r(2\zeta+2)}\inf_{x\in X}\left(\langle |B|^4\hat{k}_x,\hat{k}_x\rangle^\frac{r}{2}-\langle |A|^4\hat{k}_x,\hat{k}_x\rangle^\frac{r}{2}\right)^2\\
    &&+\frac{1}{2\zeta+2}\textbf{ber}^r(|A|^2|B|^2),
    \end{eqnarray*}
    where $r\in\mathbb{N}$, $m\geq1$  and $\zeta\geq0$.
\end{theorem}

\begin{proof}
Let $\hat{k}_x$ be a normalized reproducing kernel of $\mathcal{H}$. Then
    \begin{eqnarray*}
    &&|\langle A^*B\hat{k}_x,\hat{k}_x\rangle|^{2r}\\
    &=&|\langle B\hat{k}_x,A\hat{k}_x\rangle|^{2r}\\
    &\leq&\left(\langle B\hat{k}_x,B\hat{k}_x\rangle\langle A\hat{k}_x,A\hat{k}_x\rangle\right)^r~~(\mbox{by the Cauchy Schwarz inequality})\\
    &=&\left(|\langle |B|^2\hat{k}_x,\hat{k}_x\rangle||\langle |A|^2\hat{k}_x,\hat{k}_x\rangle|\right)^r\\
    &=&\left(|\langle |B|^2\hat{k}_x,\hat{k}_x\rangle||\langle \hat{k}_x,|A|^2\hat{k}_x\rangle|\right)^r\\
    &\leq&\frac{2\zeta+1}{2\zeta+2}\langle |B|^2\hat{k}_x,|B|^2\hat{k}_x\rangle^\frac{r}{2}\langle |A|^2\hat{k}_x,|A|^2\hat{k}_x\rangle^\frac{r}{2}+\frac{1}{2\zeta+2}|\langle |B|^2\hat{k}_x,|A|^2\hat{k}_x\rangle|^r~~(\mbox{using Lemma \ref{estimation1}})\\
    &=&\frac{2\zeta+1}{2\zeta+2}\langle |B|^4\hat{k}_x,\hat{k}_x\rangle^\frac{r}{2}\langle |A|^4\hat{k}_x,\hat{k}_x\rangle^\frac{r}{2}+\frac{1}{2\zeta+2}|\langle |A|^2|B|^2\hat{k}_x,\hat{k}_x\rangle|^r\\
    &\leq&\frac{2\zeta+1}{2\zeta+2}\left(\frac{\langle |B|^4\hat{k}_x,\hat{k}_x\rangle^m+\langle |A|^4\hat{k}_x,\hat{k}_x\rangle^m}{2}\right)^\frac{r}{m}-\frac{2\zeta+1}{2^r(2\zeta+2)}\left(\langle |B|^4\hat{k}_x,\hat{k}_x\rangle^\frac{r}{2}-\langle |A|^4\hat{k}_x,\hat{k}_x\rangle^\frac{r}{2}\right)^2\\
    &&+\frac{1}{2\zeta+2}|\langle |A|^2|B|^2\hat{k}_x,\hat{k}_x\rangle|^r~~(\mbox {using \eqref{a1/2b1/2}})\\
    &\leq&\frac{2\zeta+1}{2\zeta+2}\left(\frac{\langle |B|^{4m}\hat{k}_x,\hat{k}_x\rangle+\langle |A|^{4m}\hat{k}_x,\hat{k}_x\rangle}{2}\right)^\frac{r}{m}-\frac{2\zeta+1}{2^r(2\zeta+2)}\left(\langle |B|^4\hat{k}_x,\hat{k}_x\rangle^\frac{r}{2}-\langle |A|^4\hat{k}_x,\hat{k}_x\rangle^\frac{r}{2}\right)^2\\
    &&+\frac{1}{2\zeta+2}|\langle |A|^2|B|^2\hat{k}_x,\hat{k}_x\rangle|^r~~(\mbox{using Lemma \ref{<Ax,x>^r}})\\
    &\leq&\frac{2\zeta+1}{2\zeta+2}\left\|\frac{|B|^{4m}+|A|^{4m}}{2}\right\|_\textbf{ber}^\frac{r}{m}-\frac{2\zeta+1}{2^r(2\zeta+2)}\inf_{x\in X}\left(\langle |B|^4\hat{k}_x,\hat{k}_x\rangle^\frac{r}{2}-\langle |A|^4\hat{k}_x,\hat{k}_x\rangle^\frac{r}{2}\right)^2\\
    &&+\frac{1}{2\zeta+2}\textbf{ber}^r(|A|^2|B|^2).
    \end{eqnarray*}
     Taking the supremum over $x\in X$, we obtain our desired result.
\end{proof}

Now, by taking $r=m=1$ in Theorem \ref{goat}, we obtain the following result.
\begin{cor}\label{cora*b}
     Let $A,B\in \mathscr{B}(\mathcal{H})$. Then, for $\zeta\geq0$,
    \begin{eqnarray*}
        \textbf{ber}^2(A^*B)&\leq&\frac{2\zeta+1}{4\zeta+4}\left\||B|^4+|A|^4\right\|_\textbf{ber}-\frac{2\zeta+1}{4\zeta+4}\inf_{x\in X}\left(\langle |B|^4\hat{k}_x,\hat{k}_x\rangle^\frac{1}{2}-\langle |A|^4\hat{k}_x,\hat{k}_x\rangle^\frac{1}{2}\right)^2\\
        &&+\frac{1}{2\zeta+2}\textbf{ber}\left(|A|^2|B|^2\right).
    \end{eqnarray*}
\end{cor}
\begin{remark}
    In \cite[Theorem 2.17]{esti} it was obtained that 
   \begin{equation}\label{goat1}
        \textbf{ber}^2(A^*B)\leq\frac{1}{2|\lambda|}\max\{1,|1-\lambda|\}\left\| |B|^4+|A|^4\right\|_\textbf{ber}+\frac{1}{|\lambda|}\textbf{ber}\left(|A|^2|B|^2\right).
    \end{equation}
    In particular, when $\lambda=2$,  
   inequality \eqref{goat1} becomes 
  \begin{equation}\label{goat2}
        \textbf{ber}^2(A^*B)\leq\frac{1}{4}\left\| |B|^4+|A|^4\right\|_\textbf{ber}+\frac{1}{2}\textbf{ber}\left(|A|^2|B|^2\right).
    \end{equation}
    Since $\frac{1}{4}\inf_{x\in X}\left(\langle |B|^4\hat{k}_x,\hat{k}_x\rangle^\frac{1}{2}-\langle |A|^4\hat{k}_x,\hat{k}_x\rangle^\frac{1}{2}\right)^2\geq 0$, so for $\zeta=0$, Corollary \ref{cora*b} provides an improvement of the inequality \eqref{goat2}.
\end{remark}
\section{Conclusion}
\textit{By incorporating interpolation paths of symmetric means and Orlicz functions, together with refined versions of Young’s inequality, we have established several new upper bounds for the Berezin number of bounded linear operators on reproducing kernel Hilbert spaces. The inclusion of these structures, along with the associated Berezin number and Berezin norm inequalities, enables us to obtain estimates that are sharper than many of the existing general bounds.
These findings point to several fruitful directions for future investigation.}\\

\textbf{Acknowledgements.} 
Mr. Sumon Ghosh would like to thank the UGC, Govt. of India, for the financial support in the form of a fellowship. Mr. Saikat Mahapatra would like to thank the UGC, Govt.
of India, for financial support in the form of a fellowship.\\

\textbf{Author Contributions:} All the authors have contributed equally to this research article.\\
 \textbf{Conflict of interest:} There is no competing interest.\\
 \textbf{Availability of data:}    Not applicable. \\

\bibliographystyle{amsplain}

\end{document}